\documentclass{amsart}
\usepackage{amsmath, amsthm, amscd, amsfonts, amssymb}

\makeatletter \oddsidemargin.9375in \evensidemargin \oddsidemargin
\newtheorem{thm}{Theorem}[section]
\newtheorem{lem}[thm]{Lemma}

\newtheorem{cor}[thm]{Corollary}
\theoremstyle{definition}
\newtheorem{den}[thm]{Definition}

\theoremstyle{remark}

\numberwithin{equation}{section}

\begin{document}

\title{ cohomological characterization of $T$-Lau product algebras }
\author{ N. RAZI AND A. POURABBAS}
\address{Faculty of Mathematics and Computer
Science, Amirkabir University of Technology, 424 Hafez Avenue,
Tehran 15914, Iran}
\email{Razina@aut.ac.ir    }
 \email{arpabbas@aut.ac.ir}



\begin{abstract}
Let $A$ and $B$ be Banach algebras and let $T$ be an algebra homomorphism from  $B$ into $A$. The Cartesian product space $A\times B$ by $T$- Lau product and $\ell^{1}$- norm becomes a Banach algebra $A\times_{T}B$. We investigate the notions such as injectivity, projectivity and flatness for the Banach algebra $A\times_{T}B$. We also characterize Hochschild cohomology for the Banach algebra $A\times_{T}B$.
\end{abstract}
\maketitle
\section{Introduction and Preliminaries}
Suppose that $A$ and $B$ are Banach algebras and $T:B\rightarrow A$ is an algebra homomorphism. Then we consider the Cartesian product space $A\times B$ with the following multiplication 
 $$(a,b)\times_{T}(c,d)=(ac+T(b)c+aT(d),bd)\ \ \ ((a,b),(c,d)\in A\times B),$$
which is denoted by $A\times_{T} B$. 
Let $\Vert T\Vert\leq 1$. Then we consider $A\times _T B$
with the following norm
$$\Vert (a,b)\Vert=\Vert a\Vert +\Vert b\Vert\hspace*{2cm} ((a,b)\in A\times _T B).$$
We note that $A\times _T B$ is a Banach algebra with this norm and it is called
$T$-Lau product algebras

Whenever the  Banach algebra $A$ is commutative, Bhatt and Dabshi \cite{Bhatt} have investigated the properties of the Banach algebra $A\times_{T} B$, such as Gelfand space, Arens regularity and amenability. 

Whenever  $A$ is unital with unit element $e$ and $\varphi:B\rightarrow \mathbb{C}$ is a character on $B$, assume  $T:B\rightarrow A$ is defined  by $T(b)=\phi (b)e$. In this case  the multiplication $\times_{T}$ corresponds with the product studied by Lau \cite{Lau}. Lau product was extended by Sangani Monfared for the general case and many basic  properties of this product are studied in\cite{Monfared}.

In the definition of $T$-Lau product, we can replace condition $\|T\|\leq 1$ with a  bounded algebra homomorphism $T$, because if we consider  the following norms
$$\|a\|_T=\|T\|\|a\|\ \ \ (a\in A)$$
$$\|b\|_T=\|T\|\|b\|\ \ \ (b\in B)$$
$$\|(a,b)\|_T=\|a\|_T+\|b\|_T\ \ \  (a,b)\in A\times_TB,$$ then all these norms are equivalent with the original norms.
Clearly all  results of  this paper hold when we consider these equivalent norms.

The authors  in \cite{Razi}  for every Banach algebras  $A$ and $B$  and for an algebra homomorphism $T:B\rightarrow A$ with  $\Vert T\Vert\leq 1$  have investigated some homological properties of $T$- Lau product algebra $A\times_{T}B$ such as approximate amenability, pseudo amenability, $\phi$-pseudo amenability, $\phi$- biflatness and $\phi$-biprojectivity and have presented the characterization of the double centralizer algebra of $A\times_{T}B$.  
 
Following \cite{Razi}, in this paper we studied the homological notions such as injectivity, projectivity and flatness for the Banach algebra $A\times_{T}B$. We also characterize the Hochschild cohomology for the Banach algebra $A\times_{T}B$.
\section{Injectivity, Flatness and Projectivity} 
Let $A$ be a Banach algebra. In this paper, the category of Banach left $A$-modules and Banach right $A$-modules is denoted by $A$-{\bf mod} and {\bf mod}- $A$, respectively. We denote by $B(E,F)$ the Banach space of all bounded operators from $E$ into $F$. In the category of $A$-mod, we denote the space of bounded morphisms from $E$ into $F$ by $_AB(E,F)$. A function $S\in B(E,F)$ is called admissible if there exists $S^\prime\in B(F,E)$ such that $S\circ S^\prime\circ S=S$. 

 A. Ya. Helemskii introduced the concepts of injectivity and flatness for Banach algebras \cite{Helemskii homology} and these concepts have been investigated for different classes of Banach modules in \cite{Dalse, Helemskii certain, Ramsden, White}.
\begin{den}
A Banach left $A$-module $K$ is called projective if for every admissible epimorphism $S:E\to F$ in $A$-{\bf mod}, the induced map $S_A:\ _AB(K,E)\rightarrow\ _AB(K,F)$ defined by $$S _A(R_A)=R_A\circ S\qquad (R_A\in _AB(K,E))$$ is surjective.
\end{den}
\begin{den}
A Banach left $A$-module $K$ is called injective if for every admissible monomorphism $S:E\to F$ in $A$-{\bf mod}, the induced map $S _A:\ _AB(F,K)\rightarrow \ _AB(E,K)$ defined by  $$S_A(R_A)=R_A\circ S\qquad (R_A\in \ _AB(F,K))$$ is surjective. 	
\end{den}
\begin{den}
A Banach left $A$-module $K$ is called flat if the dual module  $K^*$ in {\bf mod}-$A$ is injective with the  action defined by
$$(f\cdot a)(x)=f(a\cdot x),$$	
where $a\in A$, $x\in K$ and $f\in K^*$.
\end{den}
Let $A, B$, and $\mathcal{C}$ be Banach algebras and let  $T:B\rightarrow A$ be an algebra homomorphism with $\|T\|\leq 1$. We note that if $A$ is a Banach left $\mathcal{C}$-module, then $A\times_{T}B$ is  a Banach left $\mathcal{C}$-module via the following action
$$c\cdot (a,b)=(c\cdot a+c\cdot T(b),0)\ \ \ ((a,b)\in A\times_{T}B, c\in \mathcal{C}).$$ Similarly  if $B$ is a Banach left $\mathcal{C}$-module, then $A\times_{T}B$ is  a Banach left $\mathcal{C}$-module via the following action $$c\cdot (a,b)=(-T(c\cdot b),c\cdot b).$$ 
\begin{thm}\label{thm1}
 Suppose that $A$ and $B$ are Banach algebras and  $T:B\rightarrow A$ is an algebra homomorphism with $\|T\|\leq 1$. Suppose that $\mathcal{C}$ is Banach algebra and $A\times_{T}B$ is injective as $\mathcal{C}$-module. Then $A$ and $B$ are injective as $\mathcal{C}$-module.
\begin{proof}
Let $A$ be a Banach left $\mathcal{C}$-module and let $F,K\in \mathcal{C}$-{\bf mod}. Suppose that  $S\in\ _\mathcal{C}B(F,K)$  is admissible and monomorphism. We will show  that the induced map $S_A:\ _\mathcal{C}B(K,A)\rightarrow \ _\mathcal{C}B(F,A)$ is onto.

We conclude that  $A\times_TB\in \mathcal{C}$-{\bf mod} via  the following action
$$c\cdot (a,b)=(c\cdot a+c\cdot T(b),0).$$
Since $A\times_{T}B$ is injective, the induced map $S_{A\times_{T}B}: \ _\mathcal{C}B(K,A\times_{T}B)\rightarrow\  _\mathcal{C}B(F,A\times_{T}B)$ is onto.
Let $\lambda\in\  _\mathcal{C}B(F,A)$ and $f\in F$. We define $\tilde{\lambda}:F\rightarrow A\times_{T}B$ by $\widetilde{\lambda}(f)=({\lambda}(f),0)$. Hence we have $\tilde{\lambda}\in\  _\mathcal{C}B(F,A\times_{T}B)$. Since $S_{A\times_{T}B}: \ _\mathcal{C}B(K,A\times_{T}B)\rightarrow  \ _\mathcal{C}B(F,A\times_{T}B)$ is onto, there exists $R_{A\times_{T}B}:K\rightarrow A\times_{T}B$ such that $R_{A\times_{T}B}(T(f))=\tilde{\lambda}(f)=({\lambda}(f),0)$. We define $R_A:K\rightarrow A$ by $R_{A}= P_{A}\circ R_{A\times_{T}B}$, where $P_{A}:A\times_{T}B\rightarrow A$ is defined by $p_A(a,b)=a$.
Clearly $R_A\in\  _\mathcal{C}B(K,A)$ and $R_A\circ T^{\prime}=\lambda$. So $A$ is injective.

For injectivity of $B$, suppose that $B$ is a Banach left $\mathcal{C}$-module, $F$ and $K\in \mathcal{C}$-{\bf mod} and $S\in\  _\mathcal{C}B(F,K)$ is admissible and monomorphism. We must show that the induced map $S_B:\ _\mathcal{C}B(K,B)\rightarrow\  _\mathcal{C}B(F,B)$ is onto. We have $A\times_{T}B\in \mathcal{C}$-{\bf mod} with the following actions
$$c\cdot (a,b)=(-T(c\cdot b),c\cdot b).$$ 
Since $A\times_{T}B$ is injective, the induced map $S_{A\times_{T}B}:\ _\mathcal{C}B(K,A\times_{T}B)\rightarrow \ _\mathcal{C}B(F,A\times_{T}B)$ is onto. Let $\mu\in \ _\mathcal{C}B(F,B)$ and $f\in F$. We define $\widetilde{\mu}:F\rightarrow A\times_{T}B$ by $\widetilde{\mu}(f)=(0,\mu(f))$. Then we have $\tilde{\mu}\in\  _\mathcal{C}B(F,A\times_{T}B)$. Since $S_{A\times_{T}B}:\ _\mathcal{C}B(K,A\times_{T}B)\rightarrow\  _\mathcal{C}B(F,A\times_{T}B)$ is onto, there exists $R_{A\times_{T}B}:K\rightarrow A\times_{T}B$ such that $R_{A\times_{T}B}(T(f))=\widetilde{\mu}=(0,\mu(f))$. We define $R_A:K\rightarrow A$ by $R_A=P_A\circ R_{A\times_{T}B}$. Clearly $R_A\in \ _\mathcal{C}B(K,B)$ and $R_A\circ S=\mu$. Hence $B$ is injective. This completes the proof.	
\end{proof}
\end{thm}
Let $A, B$ and $\mathcal{C}$ be Banach algebras and let $T:B\rightarrow A$ be an algebra homomorphism with $\|T\|\leq 1$. We note that if $A\times_{T}B$ is the Banach left $\mathcal{C}$-module, then $A$ and $B$ can be the Banach left $\mathcal{C}$-modules with the following actions
$$c\cdot a=c\cdot (a,0)\quad \hbox{and}\quad c\cdot b=c\cdot(0,b),$$
where $c\in \mathcal{C}, a\in A$ and $ b\in B$.
\begin{thm}
Suppose that $A$ and $B$ are Banach algebras and $T:B\rightarrow A$ is an algebra homomorphism with $\|T\|\leq 1$. Suppose that $\mathcal{C}$ is a Banach algebra and $A$ and $B$ are injective as $\mathcal{C}$-{\bf mod}. Then $A\times_{T}B$ is injective as $\mathcal{C}$-{\bf mod}.
\begin{proof}
Let $A\times_{T}B$ be a Banach left $\mathcal{C}$-module and $F,K\in\  \mathcal{C}$-{\bf mod}. Let $S\in\  _\mathcal{C}B(F,K)$ such that $S$ is admissible and monomorphism. We must show that the induced map $S_{A\times_{T}B}:\ _\mathcal{C}B(K,A\times_{T}B)\rightarrow\  _\mathcal{C}B(F,A\times_{T}B)$ is onto.
We have $A,B\in \mathcal{C}$-{\bf mod} with the following actions
$$c\cdot a=c\cdot (a,0)\quad \hbox{and}\quad c\cdot b=c\cdot(0,b),$$
where $c\in \mathcal{C}, a\in A$ and $ b\in B$. Since $A$ and $B$ are injective, the induced maps $S_A: \ _\mathcal{C}B(K,A)\rightarrow\  _\mathcal{C}B(F,A)$ and $S_B:\ _\mathcal{C}B(K,B)\rightarrow \ _\mathcal{C}B(F,B)$
are onto. Suppose that $\lambda\in \ _\mathcal{C}B(F, A\times_TB)$ and $(a,b)\in A\times_TB$ such that $\lambda(f)=(a,b)$ for $f\in F$.

We define $\widetilde{\lambda}:F\rightarrow A$ by $\widetilde{\lambda}(f)=a+T(b)$ and $\widetilde{\mu}:F\rightarrow B$ by $\widetilde{\mu}(f)=b$. Hence we have $\widetilde{\lambda}\in\  _\mathcal{C}B(F,A)$ and $\widetilde{\mu}\in\  _\mathcal{C}B(F,B)$. Since $S_A:\ _\mathcal{C}B(k,A)\rightarrow\  _\mathcal{C}B(F,A)$ and $S_B:\ _\mathcal{C}B(K,B)\rightarrow \ _\mathcal{C}B(F,B)$ are onto, there exist $R_A:K\rightarrow A$ and $R_B:K\rightarrow B$ such that $R_A\circ S(f)=\widetilde{\lambda}(f)=a+T(b)$ and $R_B\circ S(f)=\widetilde{\mu}(f)=b$.

We define $R_{A\times_{T}B}:K\rightarrow A\times_{T}B$ by $R_{A\times_{T}B}=q_A\circ R_A+\eta_B\circ R_B$, where $\eta_B:B\rightarrow A\times_{T}B$ such that $\eta_B(b)=(-T(b),b)$.
Clearly $R_{A\times_{T}B}\in _\mathcal{C}B(K,A\times_{T}B)$ and $R_{A\times_{T}B}\circ S=\lambda$. So $A\times_{T}B$ is injective.
\end{proof}
\end{thm}
Let $A, B$ and  $\mathcal{C}$ be Banach algebras. We note that if $A^*\times B^*$ is a Banach left $\mathcal{C}$-module, then $A^*$ and $B^*$ can be consider as Banach left $\mathcal{C}$-modules via the following actions
$$c\cdot a^*=c\cdot (a^*,0)\quad\hbox{and}\quad c\cdot b^*=c\cdot(0,b^*),$$ 
where $c\in \mathcal{C}, a^*\in A^*$ and $b^*\in B^*$.
\begin{thm}
Suppose that $A$ and $B$ are Banach algebras and $T:B\rightarrow A$ is an algebra homomorphism with $\|T\|\leq 1$. Suppose that $\mathcal{C}$ is a Banach algebra. Then $A\times_{T}B$ is flat as {\bf mod}-\,$\mathcal{ C}$ if and only if $A$ and $B$ are flat as {\bf mod}-\,$\mathcal{ C}$.
\begin{proof}
Let $A\times_{T}B$ be flat as {\bf mod}-$\mathcal{C}$. With a simple argument we can show that $(A\times_{T}B)^*\cong A^*\times B^*$. Hence by similar argument as in Theorem \ref{thm1}, one can show that $A$ and $B$ are flat Banach algebras as {\bf mod}-$\mathcal{C}$

Conversely, let $A$ and $B$ be flat Banach algebras as {\bf mod}-$\mathcal{C}$, let $F,K\in \mathcal{C}$-{\bf mod} and let $S \in \ _\mathcal{C}B(F,K)$ such that $S$ is admissible and monomorphism. Then we show that the induced map $S_{A^*\times B^*}:\ _\mathcal{C}B(K,A^*\times B^*)\rightarrow \ _\mathcal{C}B(F,A^*\times B^*)$ is onto. We have $A^*,B^*\in \mathcal{C}$-{\bf mod} with the following actions
$$c\cdot a^*=c\cdot (a^*,0)\quad \hbox{and}\quad c\cdot b^*=c\cdot (0,b^*),$$
where $c\in \mathcal{C}, a^*\in A^*$ and $ b^*\in B^*$.

Since $A^*$ and $B^*$  are injective as left $\mathcal{C}$-module, the induced maps $S_{A^*}:\ _\mathcal{C}B(K,A^*)\rightarrow \ _\mathcal{C}B(F,A^*) $ and $S_{B^*}:\ _\mathcal{C}B(K,B^*)\rightarrow \ _\mathcal{C}B(F,B^*)$ are onto. Suppose that $\lambda^*\in \ _\mathcal{C}B(F,A^*\times B^*)$ and $(a^*,b^*)\in A^*\times B^*$ such that $\lambda^*(f)=(a^*,b^*)$ for $f\in F$.

We define $\widetilde{\lambda^*}:F\rightarrow A^*$ by $\widetilde{\lambda^*}(f)=a^*$ and
$\tilde{\mu^*}:F\rightarrow B^*$ by $\widetilde{\mu^*}(f)=b^*$.
Hence we have $\widetilde{\lambda^*}\in \ _\mathcal{C}B(F,A^*)$ and $\widetilde{\mu^*}\in \ _\mathcal{C}B(F,B^*)$. Since $S_{A^*}:\ _\mathcal{C}B(K,A^*)\rightarrow \ _\mathcal{C}B(F,A^*)$ and $S_{B^*}:\ _\mathcal{C}B(K,B^*)\rightarrow \ _\mathcal{C}B(F,B^*)$ are onto, there exist $R_{A^*}:K\rightarrow A^*$ and $R_{B^*}:K\rightarrow B^*$
such that $R_{A^*}\circ S(f)=\widetilde{\lambda^*}(f)=a^*$ and $R_{B^*}\circ S(f)=\widetilde{\mu^*}(f)=b^*$.

We define $R_{A^*\times B^*}:K\rightarrow A^*\times B^*$ by $R_{A^*\times B^*}=q_{A^*}\circ R_{A^*}+q_{B^*}\circ R_{B^*}$, where $q_{A^*}:A^*\rightarrow A^*\times B^*$ and $q_B^*:B^*\rightarrow A^*\times B^*$ are defined by $q_{A^*}(a^*)=(a^*,0)$ and $q_{B^*}(b^*)=(0,b^*)$, respectively. Clearly $R_{A^*\times B^*}\in \ _\mathcal{C}B(K,A^*\times B^*)$ and $R_{A^*\times B^*}\circ S=\lambda^*$. So $A^*\times B^*$ is injective as left $\mathcal{C}$-module. This completes the proof. 
\end{proof}
\end{thm}
Let $A, B$ and  $\mathcal{C}$ be Banach algebras and let $T:B\rightarrow A$ be an algebra homomorphism with $\|T\|\leq 1$.  If $A\times_T B$ is a Banach left $\mathcal{C}$-module,  as we have seen before, $A$ and $B$ are Banach left $\mathcal{C}$-modules.
\begin{thm}
Suppose that $A$ and $B$ are Banach algebras and $T:B\rightarrow A$ is an algebra homomorphism with $\|T\|\leq 1$. Suppose that $\mathcal{C}$ is a Banach algebra. Then $A\times_{T}B$ is projective as $\mathcal{C}$-{\bf mod} if and only if $A$ and $B$ are projective as $\mathcal{C}$-{\bf mod}.
\begin{proof}
Let $A\times_{T}B$ be projective as Banach left $\mathcal{C}$-module and $F,K\in \mathcal{C}$-{\bf mod}. Let $S \in \ _\mathcal{C}B(K,F)$ be admissible and epimorphism. We show that the induced map $S_{A\times_{T}B}:\ _\mathcal{C}B(A\times_{T}B, K)\rightarrow \ _\mathcal{C}B(A\times_{T}B, F)$ is onto. 

Since $A,B\in\mathcal{C}$-{\bf mod}  are projective, the induced map $S_A:\ _\mathcal{C}B(A,K)\rightarrow\ _\mathcal{C}B(A,F)$ and $S_B:\ _\mathcal{C}B(B,K)\rightarrow\ _\mathcal{C}B(B,F)$ are onto. Let $\lambda\in \ _\mathcal{C}B(A\times_{T}B, F)$ and $f_1, f_2\in F$ such that $\lambda(a,0)=f_1, \lambda(0,b)=f_2$ for $(a,0), (0,b)\in A\times_TB$. 
We define $\widetilde{\lambda}:A\rightarrow F$ by $\widetilde{\lambda}(a)=\lambda(a,0)=f_1$ and $\widetilde{\mu}:B\rightarrow F$ by $\widetilde{\mu}(b)=\lambda(0,b)=f_2$. Hence we have $\widetilde{\lambda}\in \ _\mathcal{C}B(A,F)$ and $\widetilde{\mu}\in\ _\mathcal{C}B(B,F)$. Since $ S_A:\ _\mathcal{C}B(A,K)\rightarrow \ _\mathcal{C}B(A,F)$ and $S_B:\ _\mathcal{C}B(B,K)\rightarrow \ _\mathcal{C}B(B,F)$ are onto, there exist $R_A\in \ _\mathcal{C}B(A,K)$ and $R_B\in \ _\mathcal{C}B(B,K)$ such that $S\circ R_A(a)=\widetilde{\lambda}(a)=f_1$ and $S \circ R_B(b)=\widetilde{\mu}(b)=f_2$. We define $R_{A\times_{T}B}: A\times_{T}B\rightarrow K$ by $R_{A\times_{T}B}=R_A\circ P_A+R_B\circ P_B$. Clearly $R_{A\times_{T}B}\in \ _\mathcal{C}B(A\times_{T}B,K)$ and $T^\prime\circ R_{A\times_{T}B}=\lambda$. Hence $A\times_{T}B$ is projective as left $\mathcal{C}$-module.

Conversely, let $A$ be a Banach left $\mathcal{C}$-module and $F,K\in \mathcal{C}$-{\bf mod} and let $S\in \ _\mathcal{C}B(K,F)$ such that $S$ be admissible and epimorphism. We show that the induced map $S_A:\ _\mathcal{C}B(A,K)\rightarrow \ _\mathcal{C}B(A,F)$ is onto. We have $A\times_{T}B\in \mathcal{C}$-{\bf mod} with the following action
$$c\cdot (a,b)=(c\cdot a+c\cdot T(b),0).$$ Since $A\times_{T}B$ is projective, the induced map $S_{A\times_{T}B}:\ _\mathcal{C}B(A\times_{T},K)\rightarrow \ _\mathcal{C}B(A\times_{T}B,F)$ is onto. Let $\lambda\in \ _\mathcal{C}B(A,F)$ and $f\in F$ such that $\lambda(a)=f$ for $a\in A$. We define $\widetilde{\lambda}:A\times_{T}B\rightarrow F$ by $\widetilde{\lambda}(a,b)=\lambda(a+T(b))$. Hence we have $\widetilde{\lambda}\in \ _\mathcal{C}B(A\times_{T}B,F)$. Since $S_{A\times_{T}B}:\ _\mathcal{C}B(A\times_{T}B,K)\rightarrow \ _\mathcal{C}B(A\times_{T}B,F)$ is onto, there exists $R_{A\times_{T}B}\in B(A\times_{T}B,K)$ such that $S\circ R_{A\times_{T}B}=\widetilde{\lambda}$. We define $R_A:A\rightarrow K$ by $R_A=R_{A\times_{T}B}\circ q_A$, where $q_A:A\rightarrow A\times_{T}B$ is defined by $q_A(a)=(a,0)$. Clearly $R_A\in \ _\mathcal{C}B(A,K)$ and $S\circ R_A=\lambda$. Hence $A$ is projective as left $\mathcal{C}$-module.

For the proof of projectivity of $B$, let $B$ be a Banach left $\mathcal{C}$-module and $F,K\in \mathcal{C}$-{\bf mod} and let $S\in \ _B(K,F)$ such that $S$ is admissible and epimorphism. We show that the induced map $S_B:\ _\mathcal{C}B(B,K)\rightarrow\ _\mathcal{C}B(B,F)$ is onto. We have $A\times_{T}B\in \mathcal{C}$-{\bf mod} with the following action
$$c\cdot (a,b)=(-T(c\cdot b),c\cdot b).$$ Since $A\times_{T}B$ is projective as left $\mathcal{C}$-module, the induced map $S_{A\times_{T}B}:\ _\mathcal{C}B(A\times_{T}B,K)\rightarrow \ _\mathcal{C}B(A\times_{T}B,F)$ is onto. Let $\lambda\in \ _\mathcal{C}B(B,F)$. We define $\widetilde{\lambda}:A\times_{T}B\rightarrow F$ by $\widetilde{\lambda}(a,b)=\lambda(b)$ for $(a,b)\in A\times_{T}B$. Hence $\widetilde{\lambda}\in \ _\mathcal{C}B(A\times_{T}B,F)$. Since $S_{A\times_{T}B}:\ _\mathcal{C}B(A\times_{T}B,K)\rightarrow\ _\mathcal{C}B(A\times_{T}B,F)$ is onto, there exists $R_{A\times_{T}B}\in \ _\mathcal{C}B(A\times_{T}B,K)$ such that $S\circ R_{A\times_{T}B}(a,b)=\widetilde{\lambda}(a,b)=\lambda(b)$. We define $R_B:B\rightarrow K$ by $R_B=R_{A\times_{T}B}\circ q_B$, where $q_B:B\rightarrow A\times_{T}B$ is defined by $q_B(b)=(0,b)$ for $b\in B$. Clearly $R_B\in \ _\mathcal{C}B(B,K)$ and $S\circ R_B=\lambda$. Hence $B$ is projective as left $\mathcal{C}$-module.  
\end{proof}
\end{thm}
\section{Hochschild cohomology for the Banach algebra $A\times_{T}B$}
The concept of Hochschild cohomology for Banach algebras has been studied by  Kamowitz\cite{Kamowitz},  Johnson\cite{Johnson,Johnson1972} and others. Recall that let $A$ be a Banach algebra and let $X$ be a Banach $A$-bimodule. We denote the space of bounded $n$-linear maps from $A$ into $X$ by $\mathcal{C}^n(A,X)$. For $T\in \mathcal{C}^n(A,X)$ we define the map $\delta^n:\mathcal{C}^n(A,X)\rightarrow \mathcal{C}^{n+1}(A,X)$ by 
\begin{equation*}
\begin{split}
(\delta^nT)(a_1,...,a_{n+1}) =&
a_1\cdot T(a_2,...,a_{n+1})\\ &\,\,+\sum^{n}_{i=1}(-1)^iT(a_1,...,a_ia_{i+1},...,a_{n+1})\\
=&(-1)^{n+1}T(a_1,...,a_n)\cdot a_{n+1}.
\end{split}
\end{equation*}
 $T$ is called an $n$-cocycle if $\delta^nT=0$ and it is called $n$-coboundary if there exists $S\in \mathcal{C}^{n-1}(A,X)$ such that $T=\delta^{n-1}S$. We denote the linear space of all $n$-cocycles by $\mathcal{Z}^n(A,X)$ and the linear space of all $n$-coboundaries by $\mathcal{B}^n(A,X)$. Clearly $\mathcal{Z}^n(A,X)$ includes $\mathcal{B}^n(A,X)$. We also recall that the $n$-th Hochschild cohomology group $\mathcal{H}^n(A,X)$ is defined by the following quotient, 
$$\mathcal{H}^n(A,X)=\frac{\mathcal{Z}^n(A,X)}{\mathcal{B}^n(A,X)},$$
for more details, see \cite{Johnson1972}. We remark that a left (right) Banach $A$-module $X$ is called left (right) essential if the linear span of $A\cdot X=\{a\cdot x: a\in A, x\in X\}$ ($X\cdot A=\{x\cdot a: x\in X, a\in A\}$) is dense in $X$. A Banach $A$-module $X$ is called essential, if it is left and right essential.

Let $A$ and $B$ be Banach algebras and $T:B\rightarrow A$ be an algebra homomorphism with $\|T\|\leq 1$. Let $E$ be  a Banach $A$-bimodule. Then $E$ is also  a Banach $B$-bimodule and a Banach $A\times_{T}B$-bimodule with the following actions, respectively
$$b\cdot x=T(b)\cdot x\quad \hbox{and} \quad x\cdot b=x\cdot T(b),$$
where $b\in B, x\in E$ and
$$(a,b)\cdot x=T(b)\cdot x\quad\hbox{and}\quad x\cdot (a,b)=x\cdot T(b),$$
where $(a,b)\in A\times_{T}B, x\in E$.
\begin{lem}
Let $E$ be an essential Banach $A\times_{T}B$-bimodule. Then $E$ is an essential Banach $A$-bimodule and an essential Banach $B$-bimodule.
\end{lem}
\begin{thm}\label{t3-2}
Let $A$ and $B$ be Banach algebras with bounded approximate identity and let $T:B\rightarrow A$ be an algebra homomorphism with $\|T\|\leq 1$. Let $E$ be an essential $A\times_{T}B$-bimodule. Then $$\mathcal{H}^1(A\times_{T}B, E^*)\simeq \mathcal{H}^1(A,E^*)\times \mathcal{H}^1(B,E^*),$$
where $\simeq$ denotes the vector space isomorphism.\end{thm}
\begin{proof} By \cite[Theorem 2.9.53]{Dalse1} we have $\mathcal{H}^1(A\times_{T}B, E^*)\simeq \mathcal{H}^1(M(A\times_{T}B),E^*)$, where $\mathcal{M}(A\times_{T}B)$ denotes the double centralizer algebra of $A\times_{T}B$. But from  \cite[Theorem 4.3]{Razi} we have  $$\mathcal{M}(A\times_{T}B)\cong \mathcal{M}(A)\times\mathcal{M}(B),$$ where $\cong$ denotes the algebra isomorphism.  
 this implies that $$\mathcal{H}^1(M(A\times_{T}B),E^*)\simeq \mathcal{H}^1(\mathcal{M}(A)\times \mathcal{M}(B),E^*),$$ 
where $\mathcal{M}(A)$  and $\mathcal{M}(B)$ denote the double centralizer algebra of $A$ and $B$, respectively.

 Hence we have
$$\mathcal{H}^1(A\times_{T}B,E^*)\simeq \mathcal{H}^1(\mathcal{M}(A)\times \mathcal{M}(B),E^*),$$ 

Using  \cite[Theorem 4]{Hochschild} we obtain  $\mathcal{H}^1(\mathcal{M}(A)\times\mathcal{M}(B),E^*)\simeq \mathcal{H}^1(\mathcal{M}(A),E^*)\times \mathcal{H}^1(\mathcal{M}(B),E^*)$, thus we have
$$\mathcal{H}^1(A\times_{T}B,E^*)\simeq \mathcal{H}^1(\mathcal{M}(A),E^*)\times \mathcal{H}^1(\mathcal{M}(B),E^*).$$
Since $E$ is an essential Banach $A$-bimodule and an essential Banach $B$-bimodule, we have  $$\mathcal{H}^1(\mathcal{M}(A),E^*)\simeq\mathcal{H}^1(A,E^*)$$ and $$\mathcal{H}^1(\mathcal{M}(B),E^*)\simeq \mathcal{H}^1(B,E^*).$$
This completes the proof. 
\end{proof}

We can extend the previous theorem for the $n$-th Hochschild cohomology for the Banach algebra $A\times_TB$.
\begin{cor}
Let $A$ and $B$ be Banach algebras with bounded approximate identity and let $T:B\rightarrow A$ be an algebra homomorphism with $\|T\|\leq 1$. Let $E$ be an essential $A\times_{T}B$-bimodule. Then $$\mathcal{H}^n(A\times_{T}B, E^*)\simeq \mathcal{H}^n(A,E^*)\times \mathcal{H}^n(B,E^*)$$ for every  $n\geq 1$.
\end{cor}
\begin{proof}
 By \cite[Lemma 3.1]{Razi} $A$ and $B$ have  bounded approximate identities if and only if $A\times_TB$ has a bounded approximate identity. Using  \cite[Theorem 2.9.54]{Dalse1} one can show that if $A\times_TB$ has a bounded approximate identity, then for every essential $A\times_TB$-bimodule $E$, we have $\mathcal{H}^n(A\times_{T}B, E^*)\simeq \mathcal{H}^n(M(A\times_TB),E^*)$. In \cite[Theorem 4.3]{Razi} the authors showed that $M(A\times_T B)\cong M(A)\times M(B)$ . On the other hand Hochschild  \cite[Theorem 4]{Hochschild} showed that $\mathcal{H}^n(M(A)\times M(B), E^*)\simeq \mathcal{H}^n(M(A),E^*)\times \mathcal{H}^n(M(B),E^*) $. Hence we have $\mathcal{H}^n(A\times_{T}B, E^*)\simeq \mathcal{H}^n(A,E^*)\times \mathcal{H}^n(B,E^*)$ where $n\geq 1$. 
\end{proof}
Note that Bhatt and Dabshi in \cite{Bhatt} showed that  $A\times_{T}B$ is amenable if and only if $A$ and $B$ are amenable. In the essential case this  is an immediate corollary of Theorem \ref{t3-2}.

\end{document}